\documentclass[12pt,oneside,reqno]{amsart}

\usepackage[top=2.5cm,bottom=2.5cm,left=3cm,right=3cm]{geometry}

\usepackage{mathtools}
\IfFileExists{newtx.sty}{
  \usepackage[libertinus]{newtx}
}{
  \usepackage{libertinus}
}
\usepackage{microtype}
\usepackage{mathrsfs}
\DeclareFontFamily{U}{rsfs}{\skewchar\font127 }
\DeclareFontShape{U}{rsfs}{m}{n}{%
  <-6.5> rsfs5
  <6.5-8> rsfs7
  <8-> rsfs10
}{}
\usepackage[mathcal]{eucal}

\usepackage{tikz-cd}
\usepackage{enumitem}
\setlist{itemsep=1pt,topsep=4pt}
\usepackage{graphicx}
\usepackage{comment}
\DeclareMathAlphabet{\mathcalCM}{OMS}{cmsy}{m}{n}

\usepackage[
  backend=biber,
  giveninits=true,
  style=alphabetic,
  backref=false,
  url=false,
  doi=false,
  isbn=false
]{biblatex}
\theoremstyle{plain}
\newtheorem{theorem}{Theorem}[section]
\newtheorem{proposition}[theorem]{Proposition}
\newtheorem{lemma}[theorem]{Lemma}
\newtheorem{corollary}[theorem]{Corollary}

\theoremstyle{definition}
\newtheorem{definition}[theorem]{Definition}

\theoremstyle{remark}
\newtheorem{remark}[theorem]{Remark}

\numberwithin{equation}{section}

\usepackage{xcolor}
\definecolor{citeblue}{rgb}{0,0.35,0.75}
\usepackage[
  colorlinks=true,
  linkcolor=citeblue,
  citecolor=citeblue,
  urlcolor=citeblue
]{hyperref}
\usepackage[nameinlink,noabbrev]{cleveref}
\Crefname{theorem}{Theorem}{Theorems}
\Crefname{conjecture}{Conjecture}{Conjectures}
\Crefname{question}{Question}{Questions}

\title{Derived equivalences and isogenies of supersingular abelian varieties}

\author{An Jiang}
\author{Zhiyuan Li}
\author{Ziwei Lu}



\thanks{Zhiyuan Li is supported by the NSFC (Grant Nos.~12171090 and 12425105) and the Shanghai Pilot Program for Basic Research (Grant No.~21TQ00).
Zhiyuan Li is also a member of LMNS.
An Jiang and Ziwei Lu are supported by the NSFC (Grant No.~12121001).}

\begin{document}

\begin{abstract}
We study (twisted) Fourier--Mukai partners of supersingular abelian varieties.
We give a simple proof that any two supersingular abelian varieties of the same dimension at least two are derived isogenous, 
answering a question raised in \cite{LLT26} and generalizing the result of \cite{LiZou2026}. 
Moreover, in every dimension greater than two, 
we construct derived-equivalent supersingular abelian varieties whose corresponding positive powers are never isomorphic.
We also give conditions under which derived equivalence between two supersingular principally polarized abelian varieties implies that their underlying abelian varieties are isomorphic.
In particular, this implication holds in dimensions at most three.
\end{abstract}

\maketitle

\section{Introduction}\label{sec:introduction}
The derived Torelli theorem of Orlov and Polishchuk characterizes derived equivalences between abelian varieties geometrically.
They assert that for two abelian varieties $X$ and $Y$ over an algebraically closed field, 
one has 
\[
D^b(X)\simeq D^b(Y)
\quad\Longleftrightarrow\quad
\begin{gathered}
\text{there exists a symplectic isomorphism}\\
X\times\widehat{X}\xrightarrow{\sim}Y\times\widehat{Y}.
\end{gathered}
\]
See \cite{Orlov2002,Polishchuk1996} for details.
Allowing Brauer twists and finite chains of equivalences leads to \emph{derived isogeny}, 
a notion introduced by Huybrechts for K3 surfaces \cite{Huybrechts2019} and extended to abelian varieties in \cite{LiZou2026}.

\begin{definition}
\label{def:derived-isogeny}
Two abelian varieties $X$ and $Y$ over a field $k$ are \emph{derived isogenous}, 
written $D^b(X)\sim D^b(Y)$,
if there exist abelian varieties
\[
X=X_0,\ X_1,\ \ldots,\ X_n=Y
\]
and classes $\alpha_i\in\operatorname{Br}(X_{i-1})$ and $\beta_i\in\operatorname{Br}(X_i)$, 
together with $k$-linear exact equivalences
\[
D^b(X_{i-1},\alpha_i)\simeq D^b(X_i,\beta_i),
\qquad 1\leq i\leq n.
\]
Here $D^b(Z,\delta)$ denotes the bounded derived category of $\delta$-twisted coherent sheaves on $Z$.
See \cite[Definition~1.1.1]{LiZou2026} for the notion of derived isogeny.
\end{definition}

Derived isogeny is certainly an equivalence relation. 
Its connection with classical isogeny is described by Li, Lu, and Tang using their twisted derived Torelli theorem \cite{LLT26}.
Recall that an isogeny is \emph{principal} if its degree is a square.
For abelian varieties $X$ and $Y$ of dimension $g\geq 2$ over an algebraically closed field $k$, their results take the following form \cite[Theorems~1.6 and~1.7]{LLT26}:
\begin{enumerate}[label=\textup{(\roman*)},leftmargin=*,itemsep=4pt]
\item If $\operatorname{char}(k)=0$, then
\[
D^b(X)\sim D^b(Y)
\quad\Longleftrightarrow\quad
X \text{ and } Y \text{ are principally isogenous}.
\]
\item If $\operatorname{char}(k)=p>2$, then
\[
\begin{gathered}
X \text{ and } Y \text{ are prime-to-}p
\text{ principally isogenous}
\\[2pt]
\Longrightarrow\quad D^b(X)\sim D^b(Y).
\end{gathered}
\]
\end{enumerate}

The prime-to-$p$ restriction leaves nontrivial $p$-primary Brauer classes beyond the scope of the twisted derived Torelli theorem.
The question of extending these results to positive characteristic,
especially to supersingular abelian varieties, was raised in \cite[Remark~1.9]{LLT26}.
In particular, one may ask whether any two supersingular abelian varieties of the same dimension at least two are derived isogenous,
a case in which inseparable principal isogenies must be taken into account.
We answer this question affirmatively via a very simple method.
For the remainder of this introduction, let $k$ be an algebraically closed field of characteristic $p>0$.

\begin{theorem}\label{thm:intro-pair}
Let $X$ be an abelian variety over $k$ of dimension $g\geq 1$.
Then $X$ is supersingular if and only if it is derived isogenous to a superspecial abelian variety.
Moreover, any two supersingular abelian varieties over $k$ of the same dimension $g\geq 2$ are derived isogenous.
\end{theorem}
The key observation is that every positive-dimensional supersingular abelian variety admits a fibration by supersingular elliptic curves over a supersingular abelian variety of lower dimension. Applying the relative Fourier--Mukai transform of Bhatt \cite[Appendix, Theorem~A.1.1]{KrashenLieblich2008} to these fibrations produces a chain of twisted derived equivalences connecting the original variety to a superspecial one.

For supersingular abelian surfaces, the description in terms of twistor lines provides a geometric approach to derived isogeny \cite{LiZou2026}. 
This suggests that a similar geometric picture may exist in higher dimensions, although we do not pursue this direction here.

We next turn from derived isogeny to derived equivalence.
For supersingular abelian surfaces $A$ and $B$ over $k$, one has
\[
D^b(A)\simeq D^b(B)
\quad\Longleftrightarrow\quad A\cong B.
\]
In higher dimensions, this no longer holds. Indeed, we have the following counterexample:

\begin{theorem}\label{thm:intro-nonde}
For every integer $g>2$, there exist supersingular abelian varieties
$A$ and $B$ over $k$ of dimension $g$ such that
\[
D^b(A)\simeq D^b(B),
\qquad
A^n\not\cong B^n
\]
for every integer $n>0$.
\end{theorem}

A very special property for supersingular abelian surface is that it always admits a principal polarization. 
So one may wonder whether derived equivalence determines the underlying supersingular abelian variety when they admit principal polarizations. 
Here we can give the following rigidity result, which is also valid in characteristic two.

\begin{theorem}\label{thm:intro-ppav}
Let $X$ and $Y$ be supersingular principally polarized abelian varieties over $k$ of the same dimension $g\geq2$.
Then
\[
D^b(X)\simeq D^b(Y)
\quad\Longleftrightarrow\quad X\cong Y
\]
provided that at least one of the following conditions holds:
\begin{enumerate}[label=\textup{(\roman*)},leftmargin=*]
\item $g=3$;
\item $X[p^\infty]$ is indecomposable, and either $g$ is odd or $p\equiv3\pmod4$;
\item $X[p^\infty]$ has the $p$-divisible group of a supersingular elliptic curve as a direct factor.
\end{enumerate}

\end{theorem}

The proof combines Orlov's description of derived equivalences with a reduced-norm criterion for supersingular abelian varieties.
The remaining even-dimensional cases are discussed in Remark~\ref{rem:even-norm}.

\subsection*{Acknowledgements}
We thank Haitao Zou, Chia-Fu Yu, and Jiangwei Xue for helpful discussions. 
We thank Haitao Zou for carefully reading
an earlier draft of this paper and for many helpful
comments and suggestions.

\subsection*{History and AI disclosure}
The main question of this paper grew out of the second author's earlier joint work with Haitao Zou \cite{LiZou2026}.
The idea of constructing derived isogenies using abelian fibrations arose during our work on the supersingular Tate conjecture for hyperk\"ahler varieties \cite{FHL26}.
This method is central to our proof.
We used LLM to help locate the relevant result in Bhatt's appendix to \cite{KrashenLieblich2008} and to improve the writing.

\section{Preliminaries and notation}\label{sec:preliminaries}

\subsection{}
Unless otherwise specified, $k$ is an algebraically closed field of characteristic $p>0$.
All varieties and morphisms are defined over $k$, and products are taken over $k$.
For an $S$-scheme $X$ and a morphism $T\to S$, we write $X_T=X\times_S T$.

\subsection{}
We write $D^b(X)$ for the bounded derived category of coherent sheaves on $X$. 
All derived equivalences are $k$-linear and exact.
For a smooth projective variety $X$, its cohomological Brauer group is
\[
\operatorname{Br}(X)
=H^2_{\mathrm{\acute{e}t}}(X,\mathbb{G}_m)_{\mathrm{tors}}.
\]
For $\alpha\in\operatorname{Br}(X)$, we denote by $D^b(X,\alpha)$ the bounded derived category of $\alpha$-twisted coherent sheaves.
An abelian variety $Y$ is a \emph{Fourier--Mukai partner} of $X$ if $D^b(X)\simeq D^b(Y)$.

\subsection{}
Let $X$ and $Y$ be abelian varieties, with duals $\widehat{X}$ and $\widehat{Y}$.
An isomorphism
\[
f=
\begin{pmatrix}
a&b\\
c&d
\end{pmatrix}
:X\times\widehat{X}\xrightarrow{\sim}Y\times\widehat{Y}
\]
is \emph{symplectic} if
\[
f^{-1}=
\begin{pmatrix}
\widehat{d}&-\widehat{b}\\
-\widehat{c}&\widehat{a}
\end{pmatrix}.
\]
Such isomorphisms are called \emph{isometric} in \cite{Orlov2002}.

\subsection{}
For a field $L$ of characteristic $p$, let
\[
\alpha_{p,L}
=\ker\bigl(F:\mathbb{G}_{a,L}\to\mathbb{G}_{a,L}^{(p)}\bigr).
\]
be additive group scheme. 
We write $\alpha_p=\alpha_{p,k}$.

\subsection{}

For an abelian variety $X$, its dual is denoted by $\widehat{X}$. 
Its $a$-number is
\[
a(X)=\dim_k\operatorname{Hom}(\alpha_p,X[p]).
\]
We denote by $M_X$ the covariant Dieudonn\'e module of $X[p^\infty]$.

\subsection{} An abelian variety is \emph{supersingular} if all slopes of its Newton polygon are equal to $1/2$.
Over $k$, this is equivalent to being isogenous to a product of supersingular elliptic curves \cite[Theorem~4.2]{Oort1974}.

\subsection{}
For a supersingular abelian variety $X$ of dimension $g$, set
\[
R_X=\operatorname{End}(X)\otimes_{\mathbb{Z}}\mathbb{Z}_p.
\]
Let $D$ be the quaternion division algebra over $\mathbb{Q}_p$.
Then
\[
R_X\otimes_{\mathbb{Z}_p}\mathbb{Q}_p\cong M_g(D).
\]
We denote the reduced norm of this algebra by $\operatorname{Nrd}$.
After extension to a splitting field $L$, it becomes the determinant under an identification
\[
M_g(D)\otimes_{\mathbb{Q}_p}L\cong M_{2g}(L).
\]
In particular,
\[
\operatorname{Nrd}([m])=m^{2g},
\qquad
\operatorname{Nrd}(f)=\deg(f)
\]
for every nonzero integer $m$ and every isogeny $f\in\operatorname{End}(X)$.
We use the same notation for its restriction
\[
\operatorname{Nrd}:R_X^\times\longrightarrow\mathbb{Z}_p^\times.
\]

\section{Proof of Theorem \ref{thm:intro-pair}}
\label{sec:proof-main}

\subsection{}
As before, $k$ is an algebraically closed field of characteristic $p>0$. The following proposition is a direct consequence of Bhatt's theorem, applied to the torsor $X\to Y$ under the abelian scheme $C\times Y\to Y$.

\begin{proposition}[{\cite[Appendix, Theorem~A.1.1]
{KrashenLieblich2008}}]
\label{prop:abelian-extension}
Let
\[
0\longrightarrow C\longrightarrow X\longrightarrow Y
\longrightarrow 0
\]
be a short exact sequence of abelian varieties over $k$.
Then there exist a class $\beta\in\operatorname{Br}(\widehat{C}\times Y)$ and a $Y$-linear Fourier--Mukai equivalence
\[
D^b(X)\simeq D^b(\widehat{C}\times Y,\beta).
\]
\end{proposition}

\subsection{}
Recall that an abelian variety over $k$ is superspecial if it is isomorphic to a product of supersingular elliptic curves.
By a theorem of Oort \cite[Theorem~4.2]{Oort1974}, every supersingular abelian variety over the algebraically closed field $k$ is isogenous to such a product, and hence to a superspecial abelian variety.

\subsection{Proof of Theorem~\ref{thm:intro-pair}}

Fix a supersingular elliptic curve $E$ over $k$.
Suppose first that $X$ is derived isogenous to a superspecial abelian variety.
By \cite[Theorem~1.1]{CanonacoStellari2007}, each twisted derived equivalence in the chain is of Fourier--Mukai type.
Hence \cite[Corollary~3.4]{Olander2026} implies that the underlying abelian varieties at each step are isogenous.
Thus $X$ is isogenous to a superspecial abelian variety and is therefore supersingular.

Conversely, we prove by induction on $g\geq 2$ that every supersingular abelian variety of dimension $g$ is derived isogenous to a product of supersingular elliptic curves.

Choose an isogeny from a product of supersingular elliptic curves to $X$.
The image of any elliptic factor is an elliptic abelian subvariety $E_1\subset X$, since the restriction has finite kernel. 
Moreover, $E_1$ is supersingular, being isogenous
to that factor.
Set $Z=X/E_1$. Then
\[
0\longrightarrow E_1\longrightarrow X\longrightarrow Z
\longrightarrow 0
\]
is exact, and $Z$ is supersingular of dimension $g-1$.
By Proposition~\ref{prop:abelian-extension}, there is a class $\beta_1\in\operatorname{Br}(\widehat{E}_1\times Z)$ such that
\[
D^b(X)\simeq D^b(\widehat{E}_1\times Z,\beta_1).
\]

If $g=2$, then $Z$ is a supersingular elliptic curve,
so $\widehat{E}_1\times Z$ is a superspecial abelian surface.
The displayed equivalence therefore proves the base case.

Now suppose that $g>2$.
By the induction hypothesis, $Z$ is derived isogenous to a product $E_2\times\cdots\times E_g$ of supersingular elliptic curves.
Taking the exterior product of the Fourier--Mukai kernels realizing this derived isogeny with the kernel of the identity functor on $\widehat{E}_1$ shows that $\widehat{E}_1\times Z$ is derived isogenous to
\[
\widehat{E}_1\times E_2\times\cdots\times E_g.
\]
By transitivity, $X$ is derived isogenous to this product,
completing the induction.

Finally,
by \cite[Corollary~6.7]{Ogus1979}, every superspecial abelian variety of dimension $g\geq 2$ over $k$ is isomorphic to $E^g$.
Thus every supersingular abelian variety of dimension $g\geq 2$ is derived isogenous to $E^g$.
Since derived isogeny is an equivalence relation, any two supersingular abelian varieties of the same dimension $g\geq 2$ are therefore derived isogenous.\qed

\section{Examples of nontrivial Fourier--Mukai partners}
\label{sec:duality-polarizations}

\subsection{} We construct supersingular abelian varieties whose positive powers are not isomorphic to the dual.
\begin{theorem}\label{thm:nonselfdual}
For every integer $g>2$, there exists a supersingular abelian variety $X/k$ of dimension $g$ such that
\[
    X^n \not\cong \widehat{X^n}
    \qquad\text{for every integer }n\geq 1.
\]
In particular, $X\not\cong \widehat{X}$.
\end{theorem}

\begin{proof}
It suffices to construct  examples over $k=\overline{\mathbb{F}}_p$. 
Choose a supersingular elliptic curve $E_0/\mathbb{F}_{p^2}$ whose $p^2$-power Frobenius endomorphism is $[-p]$, and set
\[
E=E_0\times_{\mathbb{F}_{p^2}}k.
\]
Every endomorphism of $E$ is defined over $\mathbb F_{p^2}$;
see \cite[Section~4]{Yu2020}.
Since $E_0$ is supersingular, we may choose an isomorphism
\[
    \alpha_{p,\mathbb F_{p^2}}
    \xrightarrow{\sim} \ker(F_{E_0})
\]
over $\mathbb F_{p^2}$.
Its base change gives an identification
\[
    \iota:\alpha_p\xrightarrow{\sim}\ker(F_E),
\]
where $F_E:E\to E^{(p)}$ is the relative Frobenius.
Note that every endomorphism of $E$ preserves $\ker(F_E)$. 
Since the endomorphism and $\iota$ are defined over $\mathbb{F}_{p^2}$, restriction induces a ring homomorphism
\[
\operatorname{End}_k(E)
\longrightarrow
\operatorname{End}_{\mathbb{F}_{p^2}}
(\alpha_{p,\mathbb{F}_{p^2}})
=\mathbb{F}_{p^2}.
\]
Every endomorphism of $E$ acts on $\ker(F_E)$ through multiplication by a scalar in $\mathbb{F}_{p^2}$.

Choose $a_1,\ldots,a_g\in k$ linearly independent over $\mathbb{F}_{p^2}$, and let $H\subset E^g$ be the subgroup scheme defined by the embedding
\[
\alpha_p\hookrightarrow E^g,\qquad
t\longmapsto
\bigl(\iota(a_1t),\ldots,\iota(a_gt)\bigr).
\]
Since $H\cong\alpha_p$, the quotient isogeny
\[
q:E^g\longrightarrow X:=E^g/H
\]
has degree $p$. 
In particular, $X$ is a supersingular abelian variety of dimension $g$.

Next, observe that any homomorphism $h:E^g\to E$ which annihilates $H$ also annihilates $\ker(F_E)^g$.
Indeed, let $h_i\in\operatorname{End}_k(E)$ be the restriction of $h$ to the $i$-th factor and we have
\[
h(x_1,\ldots,x_g)=\sum_{i=1}^g h_i(x_i).
\]
The restriction of $h_i$ to $\ker(F_E)$ is multiplication by some $b_i\in\mathbb F_{p^2}$.
Thus $h|_H=0$ implies
\[
    \sum_{i=1}^g b_i a_i=0.
\]
The linear independence of the $a_i$ forces every $b_i$ to vanish.

Now fix $n\geq 1$ and suppose that there is an isomorphism
\[
    u:X^n\xrightarrow{\sim}\widehat{X^n}.
\]
Let $q_n=q^{\times n}:E^{gn}\to X^n$, and let $\lambda_n:E^{gn}\xrightarrow{\sim}\widehat{E^{gn}}$ be the product principal polarization.
Then
\[
    f=\lambda_n^{-1}\circ\widehat{q_n}\circ u\circ q_n
\]
is an endomorphism of $E^{gn}$. 
It is an isogeny of degree
\[
    \deg(f)=\deg(q_n)^2=p^{2n},
\]
and it annihilates $\ker(q_n)=H^n$.

Write $f=(f_{ij})$, where
\[
f_{ij}\colon E^g\longrightarrow E,
\qquad 1\leq i\leq gn,\quad 1\leq j\leq n,
\]
are its component homomorphisms with respect to
$(E^g)^n$.
Since $f$ annihilates $H^n$, each $f_{ij}$ annihilates $H$.
By the preceding observation, each $f_{ij}$ therefore annihilates $\ker(F_E)^g$.
Consequently,
\[
\ker(F_E)^{gn}\subseteq\ker(f).
\]
Comparing the orders of these finite group schemes gives
\[
p^{gn}\mid\deg(f)=p^{2n},
\]
which contradicts $g>2$.

\end{proof}

\begin{remark}
Thm~\ref{thm:nonselfdual} constructs supersingular abelian varieties of dimension $g$, for every $g>2$, that do not admit principal polarizations.
As Chia-Fu Yu pointed out to us, the existence of such varieties in every dimension at least three also follows from a dimension comparison.
Indeed, unpolarized supersingular isogeny leaves in dimension $d$ have dimension $d(d-1)/2$, whereas the supersingular locus in the moduli space of principally polarized abelian varieties has dimension $\lfloor d^2/4\rfloor$; see \cite[Sections~6--7]{Oort2009Foliations}.
Since
\[
\left\lfloor\frac{d^2}{4}\right\rfloor
<
\frac{d(d-1)}{2}
\qquad(d\geq3),
\]
principal polarizability imposes a genuine restriction in higher dimensions.
\end{remark}

\subsection{Proof of Theorem \ref{thm:intro-nonde}} Combining Theorem~\ref{thm:nonselfdual} with Mukai's equivalence \cite[Proposition~2.4]{Orlov2002}
produces nonisomorphic derived-equivalent examples in every dimension greater than two. 
This proves Theorem \ref{thm:intro-nonde}.

\subsection{} Let us consider some application of  Theorem~\ref{thm:nonselfdual}. 
Using Zarhin's trick, we can show that admitting a principal polarization is not preserved by derived equivalence.

\begin{proposition}\label{prop:nonppav}
   For every integer $g>2$, there exist supersingular abelian varieties $X$ and $Y$ of dimension $8g$ such that
\[
D^b(X)\simeq D^b(Y),
\]
where $X$ admits a principal polarization and $Y$ does not.
\end{proposition}
  \begin{proof}
      By Theorem~\ref{thm:nonselfdual}, there exists a supersingular abelian variety $A$ of dimension $g$ such that
\[
A^n\not\cong\widehat{A^n}
\qquad\text{for every }n\geq1.
\]
Set
\[
X=(A\times\widehat{A})^4,
\qquad Y=A^8.
\]
Both are supersingular of dimension $8g$, and Zarhin's trick gives a principal polarization on $X$; see \cite[Theorem~1.1(ii) and Remark~1.4]{Zarhin2023}.

Mukai's equivalence
\[
D^b(\widehat{A})\simeq D^b(A)
\]
and the identity equivalence on $D^b(A)$ induce, by taking exterior products of their Fourier--Mukai kernels, an equivalence
\[
D^b(X)
\simeq D^b(A^4\times A^4)
=D^b(Y).
\]
On the other hand, we have 
\[
Y=A^8\not\cong\widehat{A^8}=\widehat{Y},
\]
which implies that $Y$ cannot admit any principal polarization.
  \end{proof}

\section{Principal polarized FM partner of supersingular ppav}

\subsection{}In this section, we consider Fourier--Mukai partners of supersingular principally polarized abelian varieties that also admit principal polarizations. 
We first recall the following standard consequence of Ogus's crystalline Torelli theorem; see
\cite[Theorem~6.2]{Ogus1979}.

\begin{lemma}\label{lem:norm}\label{lem:elliptic-factor}
Let $X$ be a supersingular abelian variety of dimension $g>1$.
Its \emph{unpolarized central leaf}, understood here as a set of isomorphism classes (cf.~\cite[Section~1.7]{Oort2009Foliations}), admits a well-defined bijection
\[
\begin{aligned}
\left\{
\begin{array}{c}
Y \text{ supersingular abelian variety}\\
Y[p^\infty]\simeq X[p^\infty]
\end{array}
\right\}\big/\cong
&\xrightarrow{\sim}
\mathbb{Z}_p^\times/\operatorname{Nrd}(R_X^\times),\\
[Y]&\longmapsto
\deg(h)\operatorname{Nrd}(R_X^\times),
\end{aligned}
\]
where $h:X\to Y$ is any isogeny of degree prime to $p$;
such an isogeny exists for every $Y$ in the set on the left. If $X[p^\infty]$ has a supersingular elliptic direct factor, then
\[
\operatorname{Nrd}(R_X^\times)=\mathbb{Z}_p^\times.
\]
In this case, every supersingular abelian variety $Y$ satisfying
$Y[p^\infty]\simeq X[p^\infty]$ is isomorphic to $X$.
\end{lemma}

We next recall a lemma that is well known to experts.
The next lemma places the principally polarized Fourier--Mukai partners under consideration in the same unpolarized central leaf and provides the degree information needed to apply
Lemma~\ref{lem:norm}.
 
\begin{lemma}\label{lem:ppav-block}\label{lem:square-block}
Let $X$ and $Y$ be supersingular principally polarized abelian varieties of dimension $g>1$. 
If $D^b(X)\simeq D^b(Y)$, then
\[
X[p^\infty]\simeq Y[p^\infty].
\]
If $X[p^\infty]$ is indecomposable, there is an isogeny $X\to Y$ of square degree prime to $p$.
\end{lemma}

\begin{proof}
By \cite[Theorem~2.19]{Orlov2002}, the equivalence induces an isomorphism
\[
f=\begin{pmatrix}a&b\\ c&d\end{pmatrix}:
X\times\widehat{X}\xrightarrow{\sim}Y\times\widehat{Y}.
\]
The principal polarizations give $X^2\cong Y^2$, and hence $M_X^{\oplus2}\simeq M_Y^{\oplus2}$.
Endomorphism rings of supersingular Dieudonn\'e modules are finite, complete $\mathbb{Z}_p$-algebras and therefore semiperfect.
Since idempotents split, Krull--Schmidt gives $M_X\simeq M_Y$; see \cite[Sections~4--5]{Krause2015}.

If $M:=M_X\simeq M_Y$ is indecomposable, then $\operatorname{End}_{\mathrm{DM}}(M)$ is local.
The matrix induced by $f$ is invertible over this ring, so at least one entry is a unit. 
Composing with the standard Fourier--Mukai transforms and using the principal polarizations permutes the block rows and columns up to signs and isomorphisms.
We may thus assume that $b:\widehat{X}\to Y$ induces an isomorphism on $p$-divisible groups. 
In particular, $b$ is an isogeny of degree prime to $p$.

Let $\mathcal E$ be the kernel of the equivalence obtained after the above modifications.
Let
\[
i:X\times\widehat X\hookrightarrow
(X\times\widehat X)\times(Y\times\widehat Y)
\]
be the graph embedding of $f$, and let
\[
q:(X\times\widehat X)\times(Y\times\widehat Y)
\longrightarrow X\times Y
\]
be the natural projection. 
Set $\pi=q\circ i$.

By \cite[Lemma~3.1]{Orlov2002}, there exists a line bundle $L$ on $X\times\widehat X$ such that
\[
\mathbf{R}\pi_*L
\cong \mathbf{R}q_*(i_*L)
\cong \mathcal E\otimes_k^{\mathbf{L}}
\mathbf{L}i_0^*(\mathcal E^\vee),
\]
where $i_0:\operatorname{Spec}k\hookrightarrow X\times Y$ is the inclusion of the origin and
\[
\mathcal E^\vee\coloneqq\mathbf{R}\mathcalCM{H}om_{X\times Y}(\mathcal E,\mathcal O_{X\times Y}).
\]

On the other hand, the morphism $\pi$ is given by
\[
\pi:X\times\widehat X\longrightarrow X\times Y,
\qquad
(x,\xi)\longmapsto (x,a(x)+b(\xi)).
\]
Since $b$ is an isogeny, $\pi$ is finite and flat of degree $\deg(b)$. 
Thus $\mathbf{R}\pi_*L\cong\pi_*L$, and taking ranks gives
\[
\begin{aligned}
\deg(b)
&=\deg(\pi)
=\operatorname{rk}(\pi_*L)\\
&=\operatorname{rk}(\mathcal E)\,
  \operatorname{rk}\bigl(\mathbf Li_0^*(\mathcal E^\vee)\bigr)\\
&=\bigl(\operatorname{rk}(\mathcal E)\bigr)^2,
\end{aligned}
\]
where the ranks of complexes are understood as virtual ranks.
Composing $b$ with the principal polarization $X\to\widehat X$ gives the required isogeny of square degree prime to $p$.
\end{proof}

\subsection{Proof of Theorem \ref{thm:intro-ppav}}
\begin{theorem}\label{thm:ppav-rigidity}
\label{thm:odd}\label{thm:even}\label{cor:elliptic-rigidity}
Let $X$ and $Y$ be supersingular principally polarized abelian varieties of dimension $g>1$. 
Assume either that
\begin{enumerate}[label=\textup{(\roman*)},leftmargin=*]
\item $X[p^\infty]$ is indecomposable, and either $g$ is odd or
$p\equiv3\pmod4$; or
\item $X[p^\infty]$ has a supersingular elliptic direct factor.
\end{enumerate}
Then
\[
D^b(X)\simeq D^b(Y)
\quad\Longleftrightarrow\quad X\cong Y.
\]
In \textup{(i)}, indecomposability holds in particular when $a(X)=1$.
\end{theorem}

\begin{proof}
Suppose that $D^b(X)\simeq D^b(Y)$.
Case \textup{(ii)} follows directly from Lemmas~\ref{lem:norm} and~\ref{lem:ppav-block}.
In case \textup{(i)}, Lemma~\ref{lem:ppav-block} gives an isogeny $h:X\to Y$ with $\deg(h)=m^2$ and $p\nmid m$.
For chosen principal polarizations, put
\[
h'=\lambda_X^{-1}\circ\widehat h\circ\lambda_Y,
\qquad s=h' h\in R_X^\times.
\]
Then $\operatorname{Nrd}(s)=m^4$.
If $g$ is odd, the unit $u=[m]s^{-(g-1)/2}$ satisfies
\[
\operatorname{Nrd}(u)
=m^{2g}(m^4)^{-(g-1)/2}=m^2.
\]
Lemma~\ref{lem:norm} therefore gives $X\cong Y$.

Suppose that $p\equiv 3\pmod 4$.
Put
\[
H=\operatorname{Nrd}(R_X^\times),\qquad
U=(\mathbb{Z}_p^\times)^2,\qquad
q=\frac{p-1}{2}.
\]
The decomposition of $p$-adic units and the $p$-adic logarithm give
\[
U=\mu_q\times(1+p\mathbb Z_p),
\]
where $\mu_q$ denotes the group of $q$-th roots of unity in $\mathbb Z_p^\times$. 
Thus, writing $U^g=\{u^g:u\in U\}$, we obtain
\[
[U:U^g]=\gcd(g,q)\,p^{v_p(g)},
\]
which is finite and odd.

For every $a\in\mathbb Z_p^\times$, the scalar unit $a\cdot\mathrm{id}\in R_X^\times$ has reduced norm $a^{2g}$.
Consequently,
\[
U^g=(\mathbb Z_p^\times)^{2g}\subseteq H\cap U,
\]
so $U/(H\cap U)$ is a quotient of $U/U^g$ and hence has odd order.
Since $m^2\in U$ and
\[
m^4=\operatorname{Nrd}(h' h)\in H,
\]
the class of $m^2$ in $U/(H\cap U)$ has order dividing $2$.
It must therefore be trivial, giving $m^2\in H$.
As $\deg(h)=m^2$, Lemma~\ref{lem:norm} implies $X\cong Y$.
\end{proof}

\begin{corollary}\label{cor:threefold}
Let $X$ and $Y$ be supersingular principally polarized abelian threefolds over $k$.
Then
\[
D^b(X)\simeq D^b(Y)
\quad\Longleftrightarrow\quad X\cong Y.
\]
\end{corollary}

\begin{proof}
If $a(X)=1$, one can apply Theorem~\ref{thm:ppav-rigidity}\textup{(i)}.
If $a(X)=2$, the Dieudonn\'e module description in \cite[Section~4]{KaremakerYobukoYu2021} shows that $X[p^\infty]$ has a supersingular elliptic direct factor.
The same holds when $a(X)=3$, since $X$ is then superspecial.
Thus Theorem~\ref{thm:ppav-rigidity}\textup{(ii)} applies in both cases.
\end{proof}

\begin{remark}\label{rem:even-norm}
Under the indecomposability hypothesis, the condition $(\mathbb{Z}_p^\times)^2\subseteq\operatorname{Nrd}(R_X^\times)$ also suffices. 
The argument does not settle the remaining cases in even dimension $g\geq4$ when $p=2$ or $p\equiv1\pmod4$.
\end{remark}

Theorem~\ref{thm:intro-ppav} follows from Theorem~\ref{thm:ppav-rigidity} and Corollary~\ref{cor:threefold}.

\printbibliography

\end{document}